\documentclass[preprint,3p,11pt]{elsarticle}

\usepackage{graphicx}
\usepackage{amsfonts}
\usepackage{amsmath}
\usepackage{amsthm}
\usepackage{subcaption}
\usepackage[inline]{enumitem}
\usepackage[list=off]{caption}
\usepackage{titlesec}

\titleformat*{\section}{\Large\bfseries}

\newtheorem{theorem}{Theorem}
\newtheorem{lemma}{Lemma}

\newcommand{\Forbidden}{L}
\newcommand{\Triangle}{C_3}
\newcommand{\Square}{C_4}
\newcommand{\KTwoThree}{K_{2,3}}
\newcommand{\House}{X}
\newcommand{\Domino}{Y}
\newcommand{\Mongolian}{Z}
\newcommand{\etal}{\mbox{\emph{et al.\ }}}

\usepackage{amssymb}

\journal{arXiv}

\begin{document}

\begin{frontmatter}



\title{Cubic Graphs with Total Domatic Number at Least Two}


\author[math]{Saieed Akbari}
\ead{s\_akbari@sharif.edu}
\author[computer]{Mohammad Motiei}
\ead{motiei@ce.sharif.edu}
\author[computer]{Sahand Mozaffari}
\ead{smozaffari@ce.sharif.edu}
\author[computer]{Sina Yazdanbod}
\ead{syazdanbod@ce.sharif.edu}
\address[math]{Department of Mathematical Sciences,}
\address[computer]{Department of Computer Engineering,\\Sharif University of Technology}

\begin{abstract}
Let $G$ be a graph.
A total dominating set of $G$ is a set $S$ of vertices of $G$ such that every vertex is adjacent to at least one vertex in $S$.
The total domatic number of a graph is the maximum number of total dominating sets which partition the vertex set of $G$. In this paper we would like to characterize the cubic graphs with total domatic number at least two.
\end{abstract}


\end{frontmatter}

\section{Introduction}
\label{sec:introduction}

The total domination of graphs was introduced by Cockayne \etal \cite{cockayne1980total}. The literature on this subject has been surveyed and detailed in the two excellent domination books by Haynes \etal \cite{haynes1998fundamentals} \cite{haynes1998fundamentals,haynes1998domination} who did an outstanding job of unifying results scattered through some 1200 domination papers at that time. More recent results on this topic can be found in \cite{henning2013total}.

Throughout this paper, all graphs are simple, that is they have no loops or multiple edges.
We denote by $V(G)$ and $E(G)$ the vertex set and the edge set of $G$, respectively.
For two positive integers $m$ and $n$, $C_n$ and $K_{m,n}$ denote the cycle of order $n$ and the complete bipartite graph of sizes $m$ and $n$, respectively.

Let $G$ be a simple graph. The \textit{open neighborhood} of a vertex $v$, denoted by $N_G(v)$, is the set of vertices adjacent to $v$ in $G$.
For a subgraph $H$ of $G$, the \textit{open neighborhood} of $H$ is defined as
$\bigcup_{v \in V(H)} \left( N_G(v) \setminus V(H) \right)$
and is denoted by $N_G(H)$.
The \textit{open neighborhood hypergraph} of $G$ is a hypergraph on the same vertex set, where the hyperedges are defined as
$\left\{ N_G(v) \ \middle|\  v \in V(G) \right\}$.

A set $T \subseteq V(G)$ of vertices is said to be a \textit{total dominating set} of $G$, if every vertex $v \in V(G)$ is adjacent to at least one vertex in $T$.
The \textit{total domatic number} of a graph $G$, denoted by $d_t(G)$, is the maximum number of total dominating sets which partition the vertex set of $G$.
One can easily see that for a graph $G$ with no isolated vertex, $d_t(G) \geq 1$. The main goal of this paper is classifying of the family of cubic graphs for which $d_t(G) \geq 2$.


It is fairly easy to see that a graph can be partitioned into two total dominating sets, if and only if its open neighborhood hypergraph is bipartite. It is an easy consequence of the renowned Lov{\'a}sz local lemma \cite{spencer1977asymptotic} that every $r$-uniform $r$-regular hypergraph is bipartite for all $r \geq 9$ \cite{spencer1987ten}. Alon and Bregman \cite{alon1988every} improved that result by employing algebraic methods to assert that the same holds for $r \geq 8$. Henning and Yeo \cite{henning20132} later adapted their technique to prove:
\begin{theorem}
\label{thm:hypergraph-bipartiteness}
\textnormal{(\cite{henning20132})} Every $r$-uniform $r$-regular hypergraph is bipartite for all $r \geq 4$.
\end{theorem}

Theorem~\ref{thm:hypergraph-bipartiteness} immediately reveals that all $r$-regular graphs for $r \geq 4$ can be partitioned into two total dominating sets. It is a well-known fact that a cycle has two disjoint total dominating sets, if and only if its order is divisible by $4$. Knowing that in the Heawood graph, complement of no total dominating set is a total dominating set \cite[p.~51, Exercise~3.14]{jukna2011extremal},
one naturally seeks to classify all cubic graphs which can be partitioned into two total dominating sets. This classification, combined with Henning and Yeo's result \cite{henning20132}, characterizes the class of regular graphs whose vertices are impossible to be decomposed into two total dominating sets. It is briefly indicated in \cite{henning20132} that there are infinitely many connected cubic graphs with total domatic number one, the characterization of which is an open problem.

Since coloring and partitioning are essentially the same, the total domination has also been studied in the literature of graph coloring under the name of coupon coloring \cite{chen2015coupon}.
A \textit{$k$-coupon coloring} of a graph $G$ is an assignment of colors from $\left\{1, \ldots, k\right\}$ to $V(G)$, such that for all $v \in V(G)$, the open neighborhood of $v$ contains all $k$ colors. The maximum $k$ for which a $k$-coupon coloring of $G$ exists is called the \textit{coupon coloring number} of $G$.
It is evident that the coupon coloring number and the total domatic number of a graph are the same.


Given a family $\mathcal{F}$ of graphs, an \textit{$\mathcal{F}$-partitioning} of the vertices of $G$ is
a set of disjoint subgraphs of $G$, each isomorphic to a graph in $\mathcal{F}$, covering every vertex of $G$.

\section{Main Result}
\label{sec:result}

In this section we prove that for all cubic regular graphs, having no subgraph isomorphic to $L$ (given in Figure~\ref{fig:forbidden}), the total domatic number is at least two.

\begin{figure}[ht]
\centering
\includegraphics[scale=0.5]{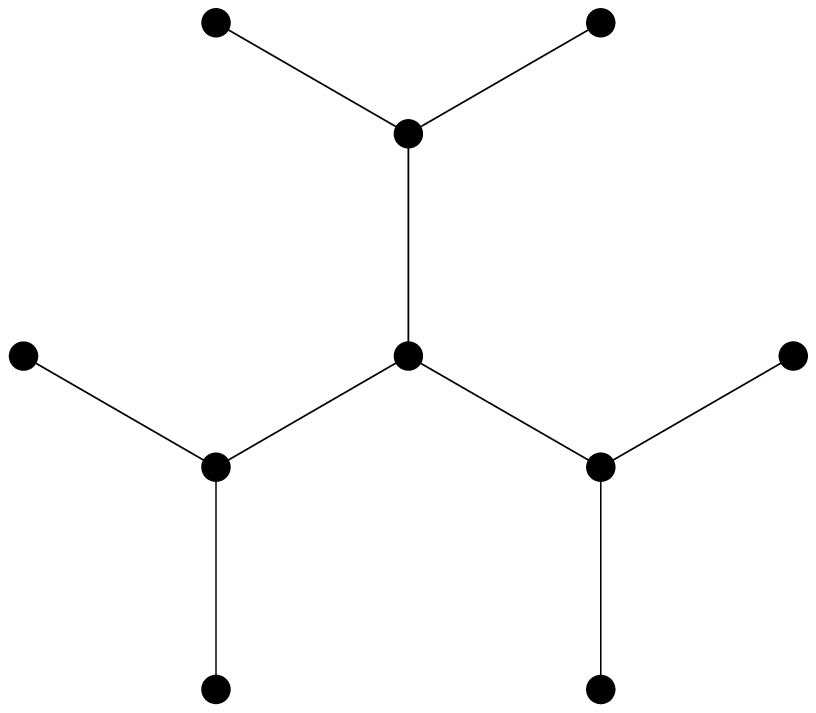}
\caption{The graph $\Forbidden$}
\label{fig:forbidden}
\end{figure}

\begin{theorem}
\label{thm:main}
The vertex set of a cubic graph can be partitioned into two total dominating sets if it has no subgraph (not necessarily induced) isomorphic to the graph $\Forbidden$. In other words, coupon coloring number of every such graph is at least two.
\end{theorem}

In this section we prove our main result towards the characterization of cubic graphs $G$ with $d_t(G) \geq 2$.
Before we embark on that, we state two relevant lemmas.

\begin{lemma}
\label{lem:forbidden}
In a cubic graph $G$ with no subgraph isomorphic to $\Forbidden$, every vertex is either contained in a $\Triangle$ or a $\Square$.
\end{lemma}


\begin{proof}
By contradiction, assume
that there exists a vertex $v \in V(G)$ which is contained neither in a $\Triangle$ nor a $\Square$. There are exactly $6$ vertices which have distance $2$ of $v$. Therefore, $G$ has a subgraph isomorphic to $\Forbidden$,
a contradiction.
\end{proof}

Figure~\ref{fig:family} presents a family $\mathcal{F}$ of graphs which will be used in the next lemma.

\begin{figure}[h!]
\centering
\hspace{\fill}
\begin{subfigure}{0.3\textwidth}
\centering
\includegraphics[scale=0.5]{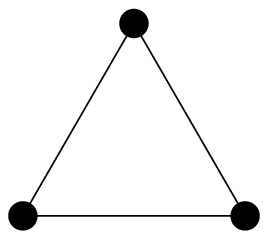}
\caption{$\Triangle$}
\end{subfigure}
\hspace{\fill}
\begin{subfigure}{0.3\textwidth}
\centering
\includegraphics[scale=0.5]{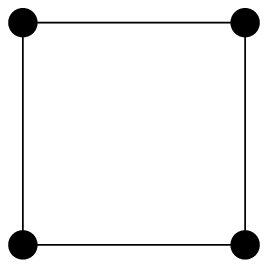}
\caption{$\Square$}
\end{subfigure}
\hspace{\fill}
\begin{subfigure}{0.3\textwidth}
\centering
\includegraphics[scale=0.5]{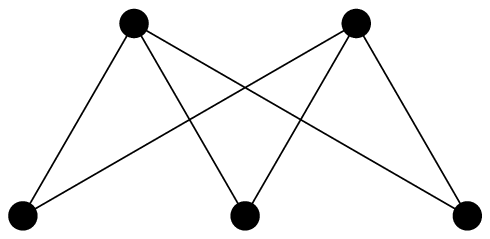}
\caption{$\KTwoThree$}
\end{subfigure}
\hspace{\fill}
\par\bigskip
\hspace{\fill}
\begin{subfigure}{0.3\textwidth}
\centering
\includegraphics[scale=0.5]{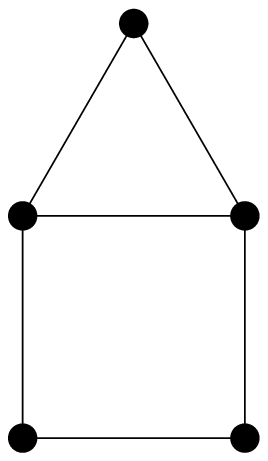}
\caption{$\House$}
\end{subfigure}
\hspace{\fill}
\begin{subfigure}{0.3\textwidth}
\centering
\includegraphics[scale=0.5]{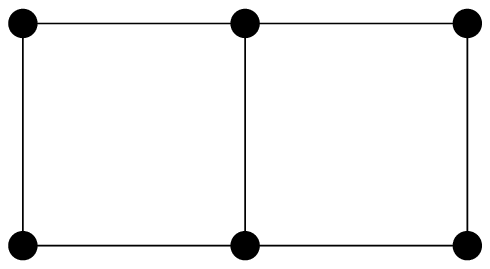}
\caption{$\Domino$}
\end{subfigure}
\hspace{\fill}
\begin{subfigure}{0.3\textwidth}
\centering
\includegraphics[scale=0.5]{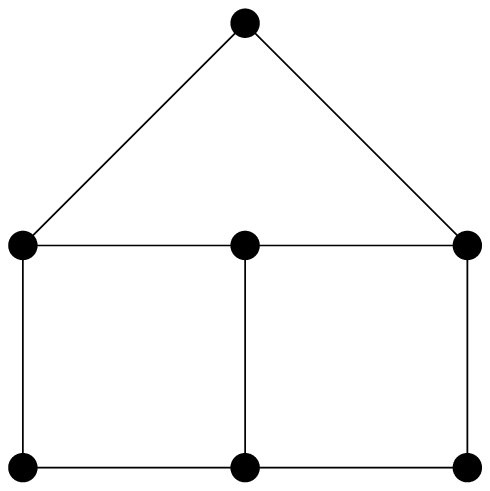}
\caption{$\Mongolian$}
\end{subfigure}
\hspace{\fill}
\caption{Family $\mathcal{F}$ of graphs}
\label{fig:family}
\end{figure}

\begin{lemma}
\label{lem:partitioning}
Every cubic graph with no subgraph isomorphic to $\Forbidden$, has an $\mathcal{F}$-partitioning,
where $\mathcal{F}$ is the family of graphs which was described in Figure~\ref{fig:family}.
\end{lemma}

\begin{proof}
Let $G$ be a cubic graph with no subgraph isomorphic to $\Forbidden$. We prove this lemma by providing an algorithm. We construct an $\mathcal{F}$-partitioning of $G$ by maintaining an $\mathcal{F}$-partitioning $\mathcal{P}$ of an induced subgraph $H$ of $G$ while iteratively adding some vertices to $H$, to obtain $H=G$.

We start with $H$ as the empty subgraph and $\mathcal{P}$ as the empty set. Assume that in the beginning of some step, $H$ is an induced proper subgraph of $G$ and $\mathcal{P}$ is an $\mathcal{F}$-partitioning of $H$. Choose a vertex $v \in V(G) \setminus V(H)$. We show how to add $v$ to $H$ and update $\mathcal{P}$.

By Lemma~\ref{lem:forbidden}, every vertex of $G$ is either contained in a $\Triangle$ or a $\Square$.
First, assume that $v$ is contained in a $\Square$ but not in a $\Triangle$. Let $C$ be that $4$-cycle and let $u_1$, $u_2$ and $u_3$ be the other vertices of $C$, where $u_1$ and $u_3$ are adjacent to $v$ in $C$.
Regarding the number of vertices among $u_1, u_2$ and $u_3$ that are already in $H$, one of the following holds:
\begin{enumerate}[leftmargin=0pt,align=left,itemindent=\parindent,label={\textbf{Case \arabic*}. }]
\item None of $u_1$, $u_2$ and $u_3$ lie in $V(H)$. In this case, we add $v, u_1, u_2$ and $u_3$ to $H$ and $C$ to $\mathcal{P}$. Clearly, $H$ has an $\mathcal{F}$-partitioning.
\item Exactly one of $u_1, u_2$ and $u_3$ are in $V(H)$.
Observe that all graphs in $\mathcal{F}$ have minimum degree $2$. Hence, a vertex in $H$ with two neighbors outside $H$, would have degree at least $4$ in $G$, a contradiction. Thus this case never occurs.
\item Exactly two of $u_1, u_2$ and $u_3$ are in $V(H)$.
By the previous observation and noting that $v$ is not on a $C_3$, these two vertices should be adjacent and they are of degree $2$ in $H$. Furthermore, they have to be in the same element $Q \in \mathcal{P}$. The only graphs in $\mathcal{F}$ with two adjacent vertices of degree $2$ are $\Triangle$, $\Square$, $\House$ and $\Domino$. Without loss of generality, assume that $u_2, u_3 \in V(H)$ (see Figure~\ref{fig:case_1_3}).
\begin{enumerate}
\item If $Q \simeq \Triangle$, then we add $v$ and $u_1$ to $H$ and replace $Q$ with $\House$, consisting of vertices of $Q$ and $C$.
\item If $Q \simeq \Square$, then we add $v$ and $u_1$ to $H$ and replace $Q$ with $\Domino$, consisting of vertices of $Q$ and $C$.
\item If $Q \simeq \House$, then we add $v$ and $u_1$ to $H$ and replace $Q$ with a $\Triangle$ and a $\Square$, consisting of vertices of $Q$ and $C$. It is clear from Figure~\ref{fig:case_1_3_c} that the subgraph induced on $V(Q) \cup \{v\}$ contains two disjoint subgraphs isomorphic to $\Triangle$ and $\Square$.
\item If $Q \simeq \Domino$, then we add $v$ and $u_1$ to $H$ and replace $Q$ with two $\Square$ graphs, consisting of vertices of $Q$ and $C$. It is clear from Figure~\ref{fig:case_1_3_d} that the subgraph induced on $V(Q) \cup \{v\}$ contains two disjoint subgraphs isomorphic to $\Square$.
\end{enumerate}

\begin{figure}[ht]
\centering
\begin{subfigure}{0.23\linewidth}
\includegraphics[scale=0.5]{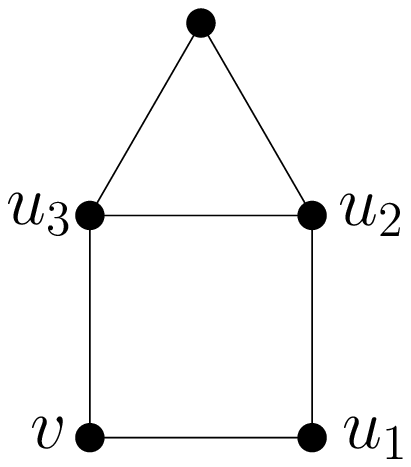}
\centering
\caption{Replacing $\Triangle$ with $\House$}
\end{subfigure}
\hspace{\fill}
\begin{subfigure}{0.23\linewidth}
\centering
\includegraphics[scale=0.5]{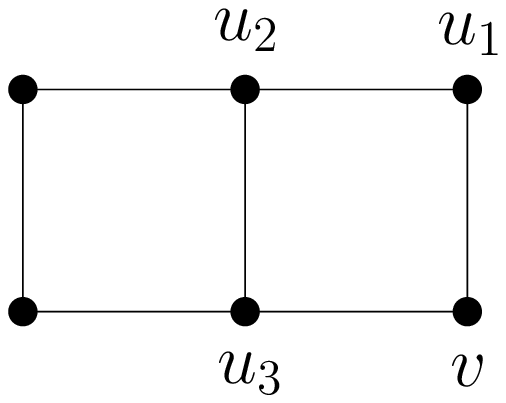}
\caption{Replacing $\Square$ with $\Domino$}
\end{subfigure}
\hspace{\fill}
\begin{subfigure}{0.23\linewidth}
\centering
\includegraphics[scale=0.5]{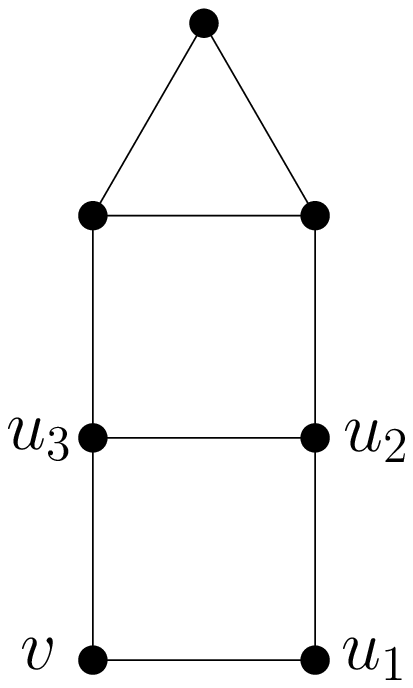}
\caption{Replacing $\House$ with a $\Triangle$ and a $\Square$}
\label{fig:case_1_3_c}
\end{subfigure}
\hspace{\fill}
\begin{subfigure}{0.23\linewidth}
\centering
\includegraphics[scale=0.5]{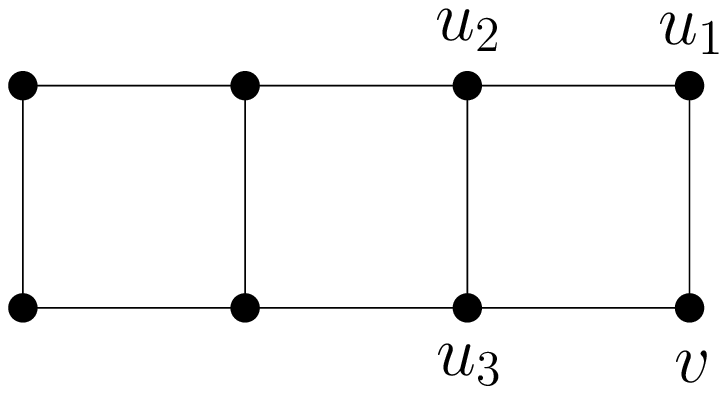}
\caption{Replacing $\Domino$ with two $\Square$ graphs}
\label{fig:case_1_3_d}
\end{subfigure}
\caption{Updating the $\mathcal{F}$-partitioning of $H$}
\label{fig:case_1_3}
\end{figure}

\item All $u_1, u_2$ and $u_3$ are in $V(H)$.
In this case,
again,
$u_1, u_2$ and $u_3$ have to be in the same element $Q$ of $\mathcal{P}$, as otherwise a vertex of degree one should appear in some element in $\mathcal{P}$. Thus $Q$ is a graph with two vertices of degree $2$ at distance $2$ of each other.
All graphs in $\mathcal{F}$ have this property, except $\Triangle$. Each of these cases are discussed below.

\begin{enumerate}
\item If $Q \simeq \Square$, then we add $v$ to $H$ and replace $Q$ with $\KTwoThree$, consisting of $v$ and all vertices of $Q$.
\item If $Q \simeq \KTwoThree$, then we add $v$ to $H$ and replace $Q$ with $\Domino$, consisting of $v$ and all vertices of $Q$.
It is clear that Figure~\ref{fig:case_1_4_b} contains a subgraph isomorphic to $\Domino$.
\item If $Q \simeq \House$, then we add $v$ to $H$ and replace $Q$ with $\Domino$, consisting of $v$ and all vertices of $Q$.
It is clear that Figure~\ref{fig:case_1_4_c} contains a subgraph isomorphic to $\Domino$.
\item If $Q \simeq \Domino$, then we add $v$ to $H$ and replace $Q$ with $\Mongolian$, consisting of $v$ and all vertices of $Q$.
\item If $Q \simeq \Mongolian$, then we add $v$ to $H$ and replace $Q$ with two disjoint $4$-cycles.
\end{enumerate}

\begin{figure}[ht]
\centering
\begin{subfigure}{0.32\textwidth}
\centering
\includegraphics[scale=0.5]{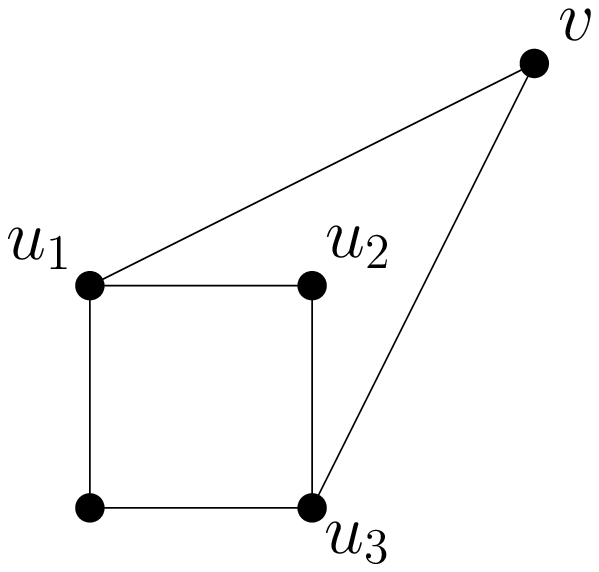}
\caption{Replacing $\Square$ with $\KTwoThree$}
\label{fig:case_1_4_a}
\end{subfigure}
\hspace{\fill}
\begin{subfigure}{0.32\textwidth}
\centering
\includegraphics[scale=0.5]{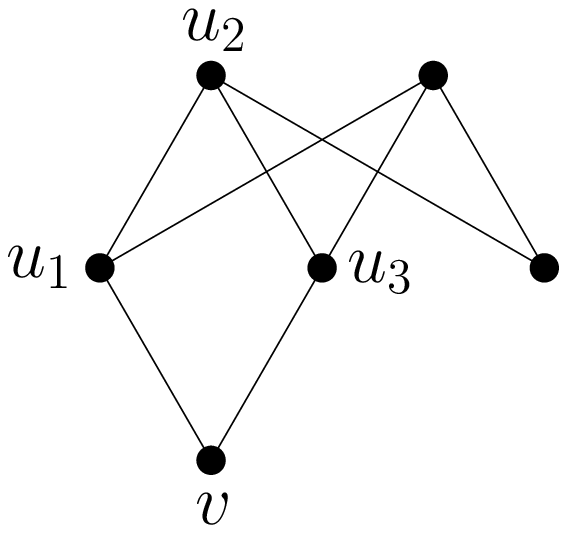}
\caption{Replacing $\KTwoThree$ with $\Domino$}
\label{fig:case_1_4_b}
\end{subfigure}
\hspace{\fill}
\begin{subfigure}{0.32\textwidth}
\centering
\includegraphics[scale=0.5]{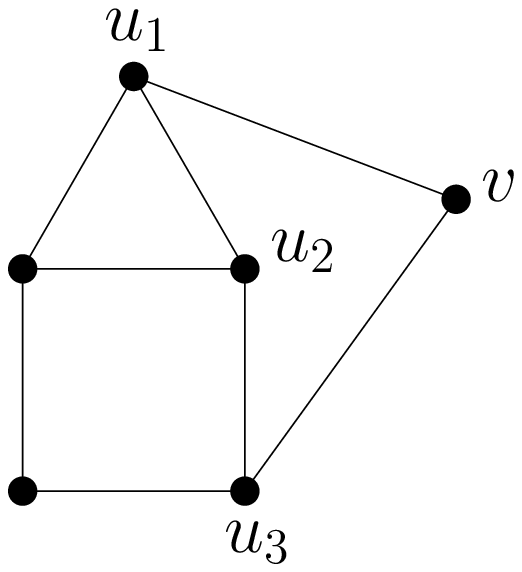}
\caption{Replacing $\House$ with $\Domino$}
\label{fig:case_1_4_c}
\end{subfigure}
\par\bigskip
\hspace{\fill}
\begin{subfigure}{0.45\textwidth}
\centering
\includegraphics[scale=0.5]{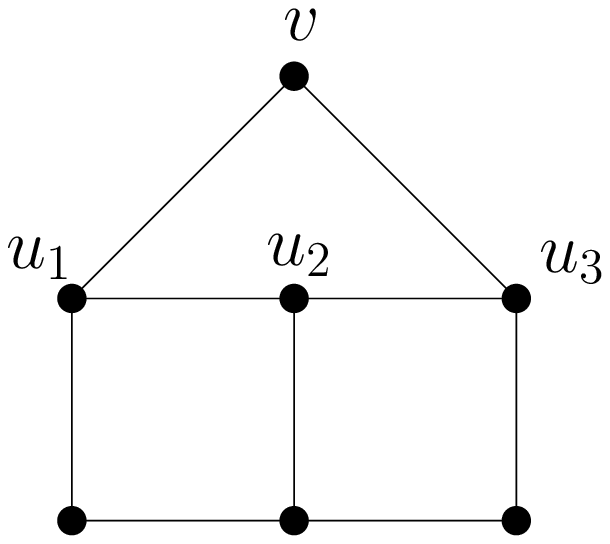}
\caption{Replacing $\Domino$ with $\Mongolian$}
\label{fig:case_1_4_d}
\end{subfigure}
\hspace{\fill}
\begin{subfigure}{0.45\textwidth}
\centering
\includegraphics[scale=0.5]{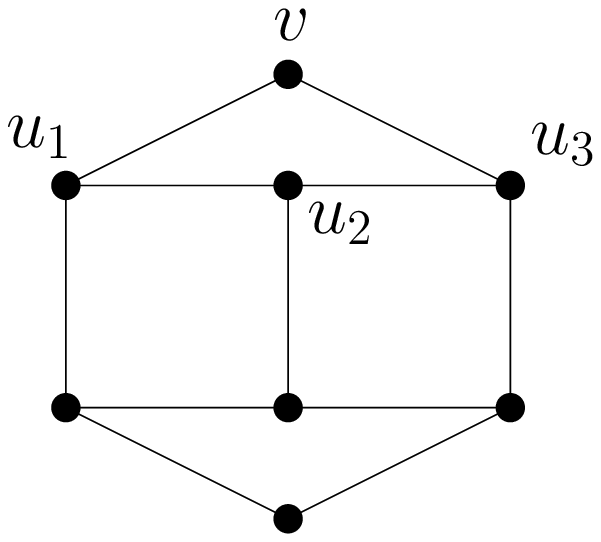}
\caption{Replacing $\Mongolian$ with two disjoint $4$-cycles}
\label{fig:case_1_4_e}
\end{subfigure}
\hspace{\fill}
\caption{Updating $\mathcal{F}$-partitioning of $H$}
\label{fig:case_1_4}
\end{figure}

Figure~\ref{fig:case_1_4} illustrates that $\mathcal{P}$ will in the end be an $\mathcal{F}$-partitioning of $H$.
\end{enumerate}

Next, assume that $v$ is contained in a $3$-cycle $C$, with vertex set $\left\{v, w_1, w_2\right\}$. Again, by observing the degrees of the graphs in $\mathcal{F}$, it can be seen that either both $w_1$ and $w_2$ are in $H$ or none are in $H$. If neither of $w_1$ and $w_2$ are in $H$, then we add $v$, $w_1$ and $w_2$ to $H$ and add $C$ to $\mathcal{P}$. If $w_1$ and $w_2$ are in $H$, then they have to be on the same element $Q \in \mathcal{P}$. Thus $Q$ is a subgraph with two adjacent vertices of degree $2$. The only graphs in $\mathcal{F}$ with this property are $\Triangle$, $\Square$, $\House$ and $\Domino$.
\begin{enumerate}[label={(\alph*) }]
\item If $Q \simeq \Triangle$, then we add all vertices of $C$ to $H$ and replace $Q$ with $\Square$, consisting of $v$ and vertices of $Q$.
\item If $Q \simeq \Square$, then we add all vertices of $C$ to $H$ and replace $Q$ with $\House$, consisting of $v$ and vertices of $Q$.
\item If $Q \simeq \House$, then we add all vertices of $C$ to $H$ and replace $Q$ with two $\Triangle$ graphs, consisting of $v$ and vertices of $Q$.
\item If $Q \simeq \Domino$, then we add all vertices of $C$ to $H$ and replace $Q$ with $\Triangle$ and a $\Square$, consisting of $v$ and vertices of $Q$.
\end{enumerate}
\vspace{-2em} 
%
\end{proof}

Using Lemma~\ref{lem:partitioning}, we can now state the proof of Theorem~\ref{thm:main}.
\begin{proof}[Proof of Theorem~\ref{thm:main}]

To prove the theorem, we show that the vertices of any cubic graph $G$, with no subgraph isomorphic to $\Forbidden$, can be colored by black and white such that each vertex is adjacent to at least one vertex of each color. By Lemma~\ref{lem:partitioning}, $G$ has an $\mathcal{F}$-partitioning $\mathcal{P}$. It is easy to verify that $C_4$, $K_{2,3}$, $X$ and $Y$ each have total domatic number at least $2$. Choose an arbitrary $2$-coupon coloring for each element $Q \in \mathcal{P}$ isomorphic to $C_4$, $K_{2,3}$, $X$ or $Y$. It now remains to color the vertices of $\Triangle$ and $\Mongolian$ subgraphs.

Figure~\ref{fig:color_C_3_&_M_2_3} shows two colorings of $\Triangle$ and $\Mongolian$, respectively.
In this figure, a filled square represents a black vertex and an empty square represents a white vertex.
To color the vertices of $\Triangle$ and $\Mongolian$ subgraphs, we repeatedly choose an arbitrary $\Triangle$ subgraph or $\Mongolian$ subgraph $Q \in \mathcal{P}$ and an arbitrary vertex $v$ in $N_G(Q)$.
If $v$ is already colored at this step, we color the vertices of $Q$ with respect to the color $v$ as in Figure~\ref{fig:color_C_3_&_M_2_3}.
However, it is possible for $v$ to be uncolored. This is the case when $v$ belongs to another $\Triangle$ subgraph or $\Mongolian$ subgraph $Q'$.
In this case we color $v$ black, and then choose the color of vertices of $Q$ and $Q'$ based on the colorings given in Figure~\ref{fig:color_C_3_&_M_2_3}.

\begin{figure}[ht]
\begin{subfigure}{0.23\textwidth}
\centering
\includegraphics[scale=0.5]{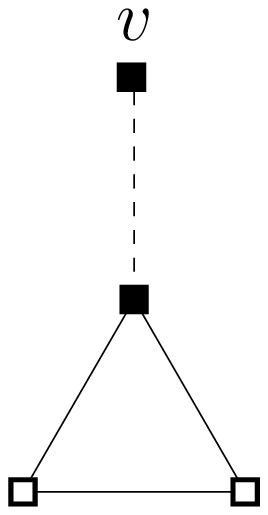}
\end{subfigure}
\hspace{\fill}
\begin{subfigure}{0.23\textwidth}
\centering
\includegraphics[scale=0.5]{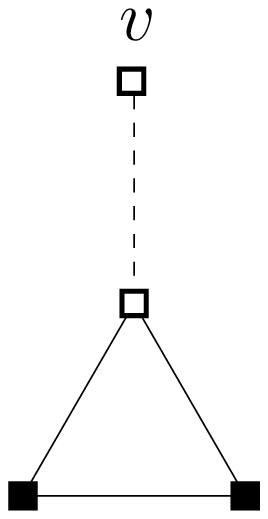}
\end{subfigure}
\hspace{\fill}
\begin{subfigure}{0.23\textwidth}
\centering
\includegraphics[scale=0.5]{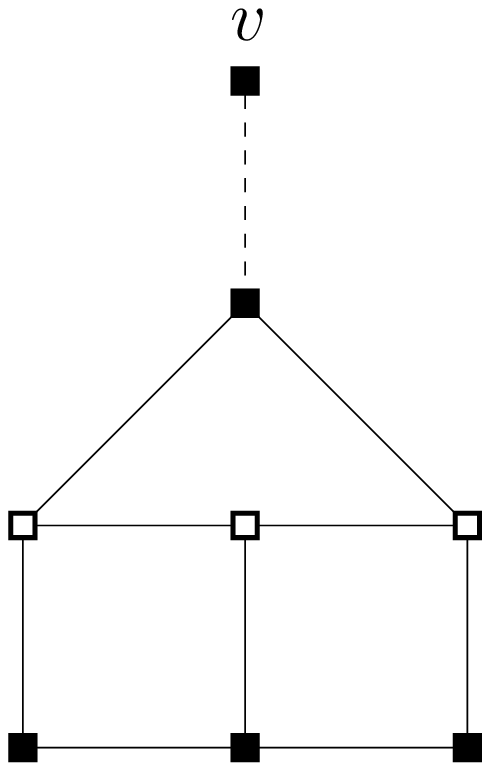}
\end{subfigure}
\hspace{\fill}
\begin{subfigure}{0.23\textwidth}
\centering
\includegraphics[scale=0.5]{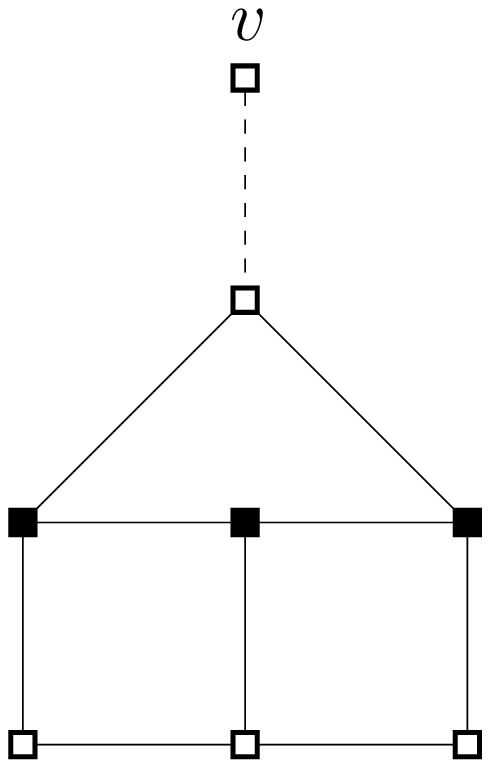}
\end{subfigure}
\caption{Colorings of $\Triangle$ and $\Mongolian$}
\label{fig:color_C_3_&_M_2_3}
\end{figure}

When every vertex is colored, each color class is a total dominating set, since any vertex of $G$ has at least one neighbor of each color. Thus we have partitioned the graph into two total dominating sets.
\end{proof}




\section*{References}

\end{document}